\documentclass[11pt]{article}
\usepackage{amsthm,amsmath,amssymb,amsfonts,mathtools}
\usepackage{geometry}
\usepackage[english]{babel}
\usepackage{microtype}
\usepackage{enumerate}
\usepackage[shortlabels]{enumitem}
\usepackage[dvipsnames]{xcolor}
\usepackage{hyperref}
\usepackage[capitalise]{cleveref}

\newtheorem{theorem}{Theorem}[section]
\newtheorem{proposition}[theorem]{Proposition}
\newtheorem{lemma}[theorem]{Lemma}
\newtheorem{claim}[theorem]{Claim}
\newtheorem{conjecture}[theorem]{Conjecture}
\numberwithin{equation}{section}

\newcommand{\ch}{\operatorname{ch}}
\newcommand{\Prob}{\mathbb P}
\newcommand{\E}{\mathbb E}

\title{The Choice Number of Random Graphs}
\author{
Lanchao Wang\thanks{School of Mathematics, Nanjing University, Nanjing, China, and ECOPRO, Institute for Basic Science, Daejeon, Korea. Email: lanchaowang@foxmail.com.}
\and
Jiasheng Zeng\thanks{Department of Mathematics, The Hong Kong University of Science and Technology, Hong Kong, China. Email: jzengbl@connect.ust.hk}
}
\date{}
\begin{document}
\maketitle

\begin{abstract}
In 1999, Alon, Krivelevich and Sudakov conjectured
that the choice number and chromatic number of
the random graph $G(n,p)$ are asymptotically equal
whenever $np\to\infty$.
We prove their conjecture, showing that, with high
probability,
$$
\ch(G(n,p))=(1+o(1))\chi(G(n,p))
$$
whenever $np\to\infty$.
This improves the bound
$\ch(G(n,p))\le(2+o(1))\chi(G(n,p))$
for $np\to\infty$ and $p\le0.9$
obtained by Krivelevich and Vu in 2001.

Our proof distinguishes two cases according to
the size of independent sets.
When independent sets of growing size are available,
we establish a sharper uniform independent-set
estimate that eliminates the factor-two loss in
the bound of Krivelevich and Vu.
When independent sets have bounded size, we
establish a deterministic result showing that,
under a suitable uniform independent-set condition,
a proper colouring with bounded colour-class sizes
yields an asymptotically matching upper bound
on the choice number.

\end{abstract}

\section{Introduction}\label{sec:main}

List colouring is a natural strengthening of ordinary graph
colouring. For a graph $G$, the chromatic number $\chi(G)$ is
the minimum number of colours needed to properly colour its
vertices. In list colouring, each vertex $v$ is assigned a
set $L(v)$ of available colours. A proper list colouring
is a proper vertex colouring $f$ satisfying
$f(v)\in L(v)$ for every $v\in V(G)$.
The \emph{choice number} $\ch(G)$ is the least integer $k$
such that $G$ admits a proper list colouring for every
assignment of lists satisfying $|L(v)|=k$ for all
$v\in V(G)$.
 Clearly, $\chi(G)\le\ch(G)$. However, these two parameters can differ significantly.
For example, Erd\H{o}s, Rubin and Taylor~\cite{ERT1980}
showed that the complete bipartite graph $K_{n,n}$
satisfies $\ch(K_{n,n})=(1+o(1))\log_2 n$.

In this paper, we study the relationship between the
chromatic number and the choice number of the random
graph $G(n,p)$. Erd\H{o}s, Rubin and Taylor~\cite{ERT1980}
asked whether $\ch(G(n,1/2))=o(n)$ with high probability.
This was answered affirmatively by Alon~\cite{Alon1992}.
Later, Alon, Krivelevich and Sudakov~\cite{AKS}
established the correct order of magnitude of the
choice number, proving that
$\ch(G(n,p))=\Theta(np/\log(np))$
with high probability whenever $2<np\le n/2$.
They further conjectured that the choice number
and the chromatic number are asymptotically equal
whenever $np\to\infty$.

\begin{conjecture}[\cite{AKS}]
Let $p=p(n)\in(0,1)$ satisfy $np\to\infty$.
Then, with high probability,
\[
\ch(G(n,p))=(1+o(1))\chi(G(n,p)).
\]
\end{conjecture}

The conjecture was already known for $p=1/2$ by
a result of Kahn; see Alon~\cite[Proposition~4.4]{Alon1993}.
As noted by Krivelevich~\cite{Krivelevich2000},
the same argument applies to every constant
$p\in(0,1)$. For varying $p$, Krivelevich and
Vu~\cite{KV} proved that
$\ch(G(n,p))\le(2+o(1))\chi(G(n,p))$
with high probability whenever $np\to\infty$
and $p\le0.9$.

Several works established the conjecture in increasingly
wide ranges of $p$. For every fixed $\delta>0$,
Krivelevich~\cite{Krivelevich2000} proved the conjecture
for $n^{-1/4+\delta}\le p\le3/4$, and
Krivelevich, Sudakov, Vu and Wormald~\cite{KSVW2003}
extended this range to
$n^{-1/3+\delta}\le p\le3/4$.
Frieze, Mitsche, P\'erez-Gim\'enez and
Pra{\l}at~\cite{FMPP2015} further extended it to
$
 n^{-1/3}(\log n)^2(\log\log n)^{1/3}
 \ll p\le1-\delta.
$
However, the conjecture remained open in its full
generality, particularly when $np$ grows arbitrarily
slowly or when $p$ approaches one.
We resolve this conjecture in its full generality.

\begin{theorem}\label{thm:main}
Let $p=p(n)\in(0,1)$ satisfy $np\to\infty$.
Then, with high probability,
\[
\ch(G(n,p))=(1+o(1))\chi(G(n,p)).
\]
\end{theorem}

Our proof distinguishes two cases according to the
size of the largest independent sets.
When $\log(1/(1-p))=o(\log(np))$, independent sets
of growing size are available.
In this case, we follow the colouring method of
Krivelevich and Vu~\cite{KV}, but establish a
sharper uniform independent-set estimate that
eliminates the factor-two loss in their bound.

In the second case, where $p$ approaches one
sufficiently rapidly that the independence number
is bounded, we take a different approach.
We establish the following deterministic result,
which may be of independent interest.
It shows that, under a suitable uniform
independent-set condition, a proper colouring with bounded colour-class sizes
yields an asymptotically matching upper bound
on the choice number. Structural conditions that
control the gap between the chromatic number and
the choice number have also been extensively studied;
see, for example,
\cite{ERT1980,Ohba2002,RSOhba2002,NRW2015,NWWZ2015}.

\begin{theorem}\label{prop:lifting}
For every integer $s\ge1$ and every $\varepsilon>0$,
there exist $\delta>0$ and $n_0$ such that the following
holds for all $n\ge n_0$.

Let $G$ be an $n$-vertex graph admitting a proper colouring
with $k$ colour classes, each of size at most $s+1$.
If every vertex set of size greater than $\delta n$
contains an independent set of size $s$, then
$\ch(G)\le k+\varepsilon n$.
\end{theorem}

The condition on large vertex subsets cannot be
omitted in general: bounded colour-class sizes alone
do not guarantee asymptotic equality.
Indeed, Kierstead~\cite{Kierstead2000} showed that
the complete $k$-partite graph with parts of size
three has chromatic number $k$ but choice number
$\lceil(4k-1)/3\rceil$.

The proof of Theorem~\ref{prop:lifting} uses the
given colouring to control how often each colour
appears in the lists, and then combines random
vertex sampling with Kahn's matching theorem~\cite{KahnLP}
to colour almost all vertices. 
The remaining vertices can be coloured at an
arbitrarily small additional cost.
To apply this result in the bounded-independence
case of Theorem~\ref{thm:main}, we show that,
with high probability, deleting $o(n)$ exceptional
vertices from $G(n,p)$ leaves a graph satisfying
its hypotheses.

\medskip
\noindent\textbf{Organization of the paper.}
In Section~\ref{sec:overview}, we give an overview of the proof.
In Section~\ref{sec:growing}, we establish a uniform
independent-set estimate and apply it to list colouring.
In Section~\ref{sec:lifting}, we prove a deterministic result
relating colourings with bounded colour classes to list colourings.
Section~\ref{sec:proof} completes the proof of
Theorem~\ref{thm:main}.

\section{Proof overview}\label{sec:overview}

The proof of Theorem~\ref{thm:main} considers two cases
according to the size of the independent sets.
When $\log(1/(1-p))=o(\log(np))$, independent sets
of growing size are available.
In the other case, $p\to1$ and we work with independent sets
of bounded size.

\medskip
\noindent\textbf{Large independent sets.}
We follow the two-stage list-colouring method of
Krivelevich and Vu~\cite{KV}.
Whenever a colour appears in sufficiently many
uncoloured lists, we use it to colour a large
independent set. Once every colour appears in
few lists, their completion lemma colours the
remaining vertices.

The main ingredient is a sharper uniform
independent-set estimate: every sufficiently large
vertex subset contains an independent set of
almost optimal size. This eliminates the factor-two
loss in the bound of Krivelevich and Vu.

To establish this estimate, we  {combine} known
chromatic-number asymptotics  {with bounded differences to}
show that the probability that a smaller random graph contains
no independent set of the required size is
exponentially small. We then use a result of
Bollob\'as to strengthen this bound sufficiently
for a union bound over all large vertex subsets.

\medskip
\noindent\textbf{Bounded independent sets.}
The preceding method relies on colouring independent
sets whose size tends to infinity.
When independent sets have bounded size, however,
even starting from a given proper colouring,
losing just one vertex from each colour class
may leave a positive proportion of the vertices
uncoloured.
Moreover, the colour classes of the given colouring
cannot necessarily be used directly, since their
vertices need not share an available colour in
their lists.

Our main ingredient is a method for covering
almost all vertices by disjoint independent sets
that can be assigned distinct available colours.
These sets need not be contained in the colour
classes of the given colouring.
This approach leads to Theorem~\ref{prop:lifting}.

Starting from a given proper vertex colouring
 {with $k$ colour classes, each of size at most $s+1$,
and from lists of $k+\lceil\varepsilon n\rceil$ colours},
we first colour and remove any of its colour
classes whose vertices have a common available
colour in their lists, deleting that colour
from the remaining lists.
This initial reduction was also used by
Reed and Sudakov~\cite{RSOhba2002}.
Once this process stops, each colour appears
in at most $s$ lists within each remaining class.
Thus, if $k'$ classes remain, each colour occurs
in at most $sk'$ lists, while every remaining
vertex still has at least
$k' {+\lceil\varepsilon n\rceil}$ available colours.

We exploit this balance by using random vertex
sampling to construct candidate independent sets
such that each vertex is well represented among
the candidates, while the overlap between them
is controlled. We then form an auxiliary hypergraph
whose edges consist of these independent sets
together with their common available colours.  {For each colour, the hypothesis on large vertex sets
allows us to partition its sampled vertices into independent
sets of size $s$ and a small number of one-vertex sets.
Assigning weights to the corresponding edges gives a
fractional matching with weighted degrees close to one
at the original vertices, while random sampling keeps
the weighted codegrees small.}
Kahn's hypergraph matching theorem~\cite{KahnLP} allows us
to find a matching covering almost all remaining
vertices, corresponding to disjoint independent
sets assigned distinct available colours.
The few vertices left uncoloured can be
 {coloured greedily with distinct colours from a small
random set reserved before the partial colouring}
at an arbitrarily small additional cost,
yielding Theorem~\ref{prop:lifting}.

Finally, we apply Theorem~\ref{prop:lifting} to $G(n,p)$.  {Choose a fixed integer $s\ge1$ such that
$(1-p)n^{2/(s+2)}\to0$ and, if $s\ge2$,
$(1-p)n^{2/s}\to\infty$.}
A first-moment estimate shows that only $o(n)$
vertices belong to independent sets of size $s+2$,
while Janson's inequality guarantees that every
vertex subset of linear size contains an independent
set of size $s$ when $s\ge2$.
After deleting the exceptional vertices, the
remaining graph has independence number at most
$s+1$ and satisfies the hypotheses of
Theorem~\ref{prop:lifting}.
Applying this theorem and restoring the deleted
vertices completes the proof.

\section{Large independent sets}\label{sec:growing}

We first treat the regime where the independent sets used as colour classes have growing size.

\begin{proposition}\label{prop:growing}
Let $p=p(n)\in(0,1)$ satisfy $d=np\to\infty$ and
$\log(1/(1-p))=o(\log(np))$.
Then, with high probability,
\[
\ch(G(n,p))=(1+o(1))\chi(G(n,p)).
\]
\end{proposition}

To prove this proposition, we use the two-stage list-colouring method of Krivelevich
and Vu~\cite{KV}. Whenever a colour appears in sufficiently
many uncoloured lists, we use it to colour a large independent
set. Once every colour appears in few lists, a completion
lemma colours the remaining vertices. The main difficulty is
to find independent sets uniformly in every sufficiently large
vertex subset, since the subsets arising in the first stage
depend on the previous colouring choices.

\subsection{Uniform independent sets}

The first stage of our list-colouring argument
requires large independent sets in every sufficiently
large vertex subset. 

\begin{lemma}\label{lem:uniform-independent}
Suppose $d=np\to\infty$ and
$\log(1/(1-p))=o(\log d)$.
Fix $\varepsilon\in(0,1)$, and put
\[
r=\left\lfloor(1-\varepsilon)
\frac{2\log d}{\log(1/(1-p))}\right\rfloor.
\]
With high probability, every vertex set of size at least
$n/(\log d)^4$ in $G(n,p)$ contains an independent set
of $r$ vertices.
\end{lemma}

We use the asymptotic formula for the chromatic number and two probability inequalities.

\begin{lemma}[{Bollob\'as~\cite{Bollobas1988},
\L uczak~\cite{Luczak1991}, and
Isaev and Kang~\cite{IK}}]\label{lem:ordinary}
Suppose $p=p(n)\in(0,1)$ satisfies $np\to\infty$ and
$\log(1/(1-p))=o(\log(np))$. Then, with high probability,
\begin{equation}\label{eq:ordinary}
\chi(G(n,p))=(1+o(1))
\frac{n\log(1/(1-p))}{2\log(np)}.
\end{equation}
\end{lemma}

\begin{lemma}[McDiarmid~{\cite[Lemma~3.2]{McDiarmid1989}}]
\label{lem:bounded-differences}
Let $F=G(m,p)$, where $m\ge1$ and $p\in[0,1]$,
and let $f$ be a real-valued function on the graphs
with vertex set $[m]$. Suppose
$|f(F_1)-f(F_2)|\le1$ whenever $F_1$ and $F_2$
differ only in edges incident with one vertex.
Then, for every $u\ge0$,
\[
\Prob\bigl(|f(F)-\E f(F)|\ge u\bigr)
\le2\exp\left(-\frac{2u^2}{m}\right).
\]
\end{lemma}

\begin{lemma}[Bollob\'as~{\cite[Theorem~5]{BollobasICM1998}}]
\label{lem:hereditary-comparison}
Let $0<p<1$, and let $r,m,t$ be integers with
$2\le r\le m\le t$. Then
\[
\Prob\bigl(\alpha(G(t,p))<r\bigr)
\le
\Prob\bigl(\alpha(G(m,p))<r\bigr)^{
\frac{t(t-1)}{m(m-1)}}.
\]
\end{lemma}
\smallskip
\begin{proof}[Proof of Lemma~\ref{lem:uniform-independent}]
Set
$
t=\left\lceil n/(\log d)^4\right\rceil,
m=\left\lfloor n/(\log d)^{10}\right\rfloor.
$
We first show that $\Prob(\alpha(G(m,p))<r)$
is exponentially small in $m$.
We then use Lemma~\ref{lem:hereditary-comparison}
to transfer this bound to $G(t,p)$, obtaining
an estimate strong enough for a union bound
over all $t$-vertex subsets.

\begin{claim}\label{clm:small-graph-tail}
For $F=G(m,p)$, we have
$
\Prob(\alpha(F)<r)\le e^{-\varepsilon^2m/16}
$
for all sufficiently large $n$.
\end{claim}

\begin{proof}
Since $mp\sim d/(\log d)^{10}\to\infty$ and
$\log(mp)\sim\log d$, the hypotheses of
Lemma~\ref{lem:ordinary} hold for $F$.
Moreover, $r\to\infty$ and
$
m/r\ge mp/(2\log d)\to\infty.
$
Since $(1-\varepsilon/2)/(1-\varepsilon)>1$,
Lemma~\ref{lem:ordinary} implies that $F$ is
$\lfloor(1-\varepsilon/2)m/r\rfloor$-colourable
with high probability.

Let $Y$ be the minimum number of vertices whose
deletion makes $F$
$\lfloor(1-\varepsilon/2)m/r\rfloor$-colourable.
Then $\Prob(Y=0)\to1$, so $\E Y=o(m)$.
Changing the edges incident with one vertex
changes $Y$ by at most one.

If $\alpha(F)<r$, every colour class after these deletions
has at most $r-1$ vertices. Hence\vspace{-0.1cm}
\[
m-Y\le
\left\lfloor\frac{(1-\varepsilon/2)m}{r}\right\rfloor
(r-1)<(1-\varepsilon/2)m,
\]
so $Y>\varepsilon m/2$.
Since $\E Y\le\varepsilon m/4$ for sufficiently
large $m$, Lemma~\ref{lem:bounded-differences}
gives
\[
\Prob(\alpha(F)<r)
\le\Prob(Y-\E Y\ge\varepsilon m/4)
\le2e^{-\varepsilon^2m/8}
\le e^{-\varepsilon^2m/16}.
\vspace{-0.6cm}\]
\end{proof}

By Claim~\ref{clm:small-graph-tail} and
Lemma~\ref{lem:hereditary-comparison},\vspace{-0.1cm}
\[
\Prob(\alpha(G(t,p))<r)
\le\exp\left(-\frac{\varepsilon^2}{16}
\frac{t(t-1)}{m-1}\right)
\le\exp\left(-\frac{\varepsilon^2}{16}
n(\log d)^2\right),
\]
where the last inequality follows from
$
t(t-1)/(m-1)\ge t^2/m\ge n(\log d)^2.
$

For every fixed $t$-vertex set $U\subseteq[n]$,
the graph $G(n,p)[U]$ has the distribution $G(t,p)$.
Taking a union bound over at most $2^n$ such sets,
the probability that some $U$ contains no independent
set of size $r$ is at most
\[
\exp\left(n\log2-\frac{\varepsilon^2}{16}
n(\log d)^2\right)=o(1).
\]
Since every vertex set of size at least
$n/(\log d)^4$ contains a $t$-vertex set,
the proof is complete.
\end{proof}

\subsection{The completion lemma}

The preceding lemma allows us to colour a large independent
set whenever a colour appears in many lists. When no such
colour remains, the following lemma 
completes the list colouring.

\begin{lemma}[Krivelevich and Vu~{\cite[Lemma~3.1]{KV}}]
\label{lem:KV}
Let $p=p(n)\in(0,1)$ satisfy $np\to\infty$,
write $d=np$, and let $G=G(n,p)$.
 With high probability, the following
holds for every vertex subset $U\subseteq V(G)$
and every list assignment on $U$.
If each vertex of $U$ has at least
$\lceil d/(\log d)^2\rceil$ colours in its list
and each colour occurs in at most
$n/(\log d)^4$ lists, then $G[U]$
has a proper list colouring.
\end{lemma}

Our formulation differs slightly from that of
\cite[Lemma~3.1]{KV}, but their proof yields
the version stated above.
Their Lemma~2.2 is used to handle an exceptional
vertex set of size at most
$B=n(\log d)^4/d^{3/2}$.
If $B<1$, this set is empty and the step can be omitted.
If $B\ge1$, then $p\le n^{-1/3+o(1)}$, so the
density assumption of their Lemma~2.2 is satisfied
and the lemma applies.
Thus the argument works whenever $d=np\to\infty$.

\subsection{Two-stage list colouring}

We now combine the independent-set estimate with
the completion lemma to prove Proposition~\ref{prop:growing}.

\begin{proof}[Proof of Proposition~\ref{prop:growing}]
Fix $\varepsilon\in(0,1)$ and let $r$ be as in
Lemma~\ref{lem:uniform-independent}.
Assume that the high-probability properties in
Lemmas~\ref{lem:ordinary}, \ref{lem:uniform-independent}
and~\ref{lem:KV} all hold for $G=G(n,p)$. We show that every list assignment
of size\vspace{-0.1cm}
\[
\left\lfloor\frac nr\right\rfloor
+\left\lceil\frac{d}{(\log d)^2}\right\rceil\vspace{-0.1cm}
\]
admits a proper colouring.

\smallskip
\noindent\textbf{Step 1: Colouring large independent sets.}
As long as some colour $c$ appears in the lists of
at least $n/(\log d)^4$ uncoloured vertices,
Lemma~\ref{lem:uniform-independent} gives an independent
set of $r$ such vertices. Colour them with $c$
and delete $c$ from every remaining list.
Each step colours $r$ vertices, so there are at most
$\lfloor n/r\rfloor$ steps.

\smallskip
\noindent\textbf{Step 2: Completing the colouring.}
When the process stops, let $U$ be the set of
uncoloured vertices. Every vertex in $U$ still has
at least $\lceil d/(\log d)^2\rceil$ available
colours, and each colour appears in fewer than
$n/(\log d)^4$ lists.
By Lemma~\ref{lem:KV}, $G[U]$ admits a proper list colouring.
Since the colours used in Step~1 were deleted from all remaining
lists, this gives a proper list colouring of $G$.
Consequently,
\[
\begin{aligned}
\ch(G(n,p))
&\le \frac nr+\frac{d}{(\log d)^2}+O(1)\\
&=\left(\frac1{1-\varepsilon}+o(1)\right)
\frac{n\log(1/(1-p))}{2\log d}
=\left(\frac1{1-\varepsilon}+o(1)\right)\chi(G(n,p)).
\end{aligned}
\]
Since $\varepsilon>0$ is arbitrary, the proof is complete.
\end{proof}
\section{From colourings to list colourings}\label{sec:lifting}

The proof of Theorem~\ref{prop:lifting} has two ingredients.
First, packings of independent sets allow us to colour almost
all vertices when each colour appears in few lists. We then use
the given colouring to obtain this condition and reserve colours
to finish the list colouring.

\subsection{A partial list-colouring lemma}

The vertices of an independent set can receive the same colour
if it belongs to all their lists. The following lemma shows that
if every large vertex set contains an independent
set of size $s$ and no colour occurs too often, then almost all
vertices can be coloured using such sets or individual vertices.

\begin{lemma}\label{lem:partial}
For every integer $s\ge1$ and $\varepsilon>0$,
there exists $\delta>0$ with the following property.
Let $N,D\ge1$ be integers, and let $G$ be an
$N$-vertex graph with lists of $D$ colours
at each vertex.

Suppose every colour occurs in at most $sD$ lists
and every vertex set of size greater than $\delta D$
contains an independent set of size $s$.
Then at least $(1-\varepsilon)N$ vertices of $G$
can be properly coloured from their lists, with
each colour used on either one vertex or an
independent set of size $s$.
\end{lemma}

For $s=1$, Hall's marriage theorem gives a proper
colouring of all vertices from their lists.
For $s\ge2$, we prove this using Kahn's hypergraph weighted matching theorem.
In our application, the auxiliary hypergraph (not necessarily uniform) $\mathcal H$
has vertex set $V(G)\cup C$, where
$C=\bigcup_{v\in V(G)}L(v)$ is the union of all colour lists.
Each hyperedge has the form $Q\cup\{c\}$, where
$Q\subseteq V(G)$ is independent, $|Q|\in\{1,s\}$,
and $c$ belongs to the list of every vertex in $Q$.
A matching selects pairwise disjoint independent sets $Q$
with distinct colours $c$. Colouring each selected $Q$
with its corresponding colour $c$ therefore gives a proper
list colouring of these vertices in $G$.
Our goal is to find a matching covering almost all vertices of $G$.

Let $w:E(\mathcal H)\to[0,1]$ assign a weight
to each hyperedge. For distinct vertices
$x,y\in V(\mathcal H)$, define
\[
d_w(x)=\sum_{e\ni x}w(e),\qquad
d_w(x,y)=\sum_{e\supseteq\{x,y\}}w(e).
\]
We call $w$ a \emph{fractional matching} if
$d_w(x)\le1$ for every $x\in V(\mathcal H)$.
We shall use the following form of Kahn's matching theorem.

\begin{theorem}[Kahn~{\cite[Theorem~1.5]{KahnLP}}]
\label{thm:kahn}
Fix an integer $b\ge1$. Let $\mathcal H$ be a
hypergraph with $|e|\le b$ for every
$e\in E(\mathcal H)$, and let $w$ be a fractional
matching. Suppose that $X\subseteq V(\mathcal H)$
satisfies $|X|\to\infty$ and
\[
\min_{x\in X}d_w(x)=1-o(1),\qquad
\max_{\substack{x,y\in V(\mathcal H), \ x\ne y}}
d_w(x,y)=o(1).
\]
Then $\mathcal H$ has a matching covering
$|X|-o(|X|)$ vertices of $X$.
\end{theorem}

The remaining task is to construct a fractional
matching satisfying the conditions of
Theorem~\ref{thm:kahn}.
A natural idea is to partition each
$S_c:=\{v\in V(G):c\in L(v)\}$ into
independent sets of size $s$ and one-vertex sets,
and assign weights to these parts.
This gives weighted degrees close to one at
the vertices of $G$, but the same pair of vertices
may appear together in many parts, resulting
in large weighted codegrees.
We overcome this difficulty by randomly sampling
vertices before forming the partitions.

\begin{proof}[Proof of Lemma~\ref{lem:partial}]
Set $C=\bigcup_{v\in V(G)}L(v)$.
If $s=1$, consider the bipartite graph with vertex classes
$V(G)$ and $C$, where
$v$ is adjacent to $c$ when $c\in L(v)$.
For every $W\subseteq V(G)$, we have\vspace{-0.1cm}
\[
D\left|\bigcup_{v\in W}L(v)\right|
\ge\sum_{c\in\bigcup_{v\in W}L(v)}
|\{v\in W:c\in L(v)\}|
=D|W|.\vspace{-0.1cm}
\]
Thus Hall's theorem
gives distinct allowed colours for all vertices.

We may therefore assume $s\ge2$ and $N>D$.
Suppose for a contradiction that the lemma fails.
Then there exist fixed $s\ge2$ and $\varepsilon_0>0$
such that, for every $\delta>0$, there is an 
$N$-vertex counterexample $G$ with a list assignment $L$.
We now let $\delta\to0$. Our hypothesis implies $\delta D\ge1$, which implies $D,N\to\infty$
as $\delta\to0$.

\smallskip
\noindent\textbf{Step 1: Sampling and partitioning.}
For each $c\in C$, put
$S_c=\{v\in V(G):c\in L(v)\}$, so $|S_c|\le sD$.
Form a random set $U\subseteq V(G)$ by including
each vertex independently with probability $\sqrt{\delta}$.
For each colour $c$, choose uniformly at random
a maximal family of pairwise disjoint independent
sets of size $s$ in $G[S_c\cap U]$.
Let $B_c$ be the set of vertices in $S_c\cap U$
not covered by these independent sets.
Since $\alpha(G[B_c])<s$,
we have $|B_c|\le\delta D$.

Let $\mathcal P_c$ be the partition of $S_c\cap U$
consisting of the selected independent sets and the
one-vertex sets from $B_c$.
Then the number of parts satisfies $|\mathcal P_c|\le |S_c\cap U|/s+\delta D$,
and taking expectations gives
\[
\E|\mathcal P_c|
\le\frac{\sqrt{\delta}|S_c|}{s}+\delta D
\le(\sqrt{\delta}+\delta)D.
\]

\smallskip
\noindent\textbf{Step 2: Constructing a fractional matching.}
Define a hypergraph $\mathcal H$ on $V(G)\cup C$
whose edges are all sets $Q\cup\{c\}$, where
$c\in C$, $Q\subseteq S_c$, and $Q$ is either
a singleton or an independent set of size $s$ in $G[S_c]$.
Assign each edge the weight\vspace{-0.1cm}
\[
w(Q\cup\{c\})
=\frac{\Prob(Q\in\mathcal P_c)}
{(\sqrt{\delta}+\delta)D}.
\]

For every $v\in V(G)$ and $c\in L(v)$,
since $\mathcal P_c$ partitions $S_c\cap U$,
the events $\{Q\in\mathcal P_c\}$ over all
$Q\ni v$ are pairwise disjoint, and their
union is $\{v\in U\}$. Hence
$
\sum_{Q\ni v}\Prob(Q\in\mathcal P_c)
=\Prob(v\in U).
$
 Consequently,\vspace{-0.1cm}
\[
\begin{aligned}
d_w(v)
=\sum_{c\in L(v)}\sum_{Q\ni v}w(Q\cup\{c\})=\frac{\sum_{c\in L(v)}\Prob(v\in U)}
{(\sqrt{\delta}+\delta)D}
=\frac1{1+\sqrt{\delta}}\ \ \ \text{ and }\ \ \ d_w(c)=\frac{\E|\mathcal P_c|}
{(\sqrt{\delta}+\delta)D}\le1.\vspace{-0.1cm}
\end{aligned}\]
Thus $w$ is a fractional matching, and
$d_w(v)=1-o(1)$ for every original vertex.

\smallskip
\noindent\textbf{Step 3: Applying Kahn's theorem.}
We verify the weighted codegree condition.
For distinct $x,y\in V(G)$, a part of
$\mathcal P_c$ can contain both vertices only if
$x,y\in U$. Hence
\[
\begin{aligned}
d_w(x,y)
&=\sum_{c\in L(x)\cap L(y)}
  \sum_{Q\supseteq\{x,y\}}w(Q\cup\{c\})=\frac{
  \sum_{c\in L(x)\cap L(y)}
  \sum_{Q\supseteq\{x,y\}}
  \Prob(Q\in\mathcal P_c)}
  {(\sqrt{\delta}+\delta)D}\\&\le\frac{|L(x)\cap L(y)|\Prob(x,y\in U)}
  {(\sqrt{\delta}+\delta)D}\le\frac{D\delta}{(\sqrt{\delta}+\delta)D}
\le\sqrt{\delta}=o(1).
\end{aligned}
\]

For $v\in V(G)$ and $c\in L(v)$, we have\vspace{-0.2cm}
\[
\begin{aligned}
d_w(v,c)
&=\sum_{Q\ni v}w(Q\cup\{c\})=\frac{\Prob(v\in U)}{(\sqrt{\delta}+\delta)D}
=\frac1{(1+\sqrt{\delta})D}
\le\frac1D=o(1).\vspace{-0.2cm}
\end{aligned}
\]
If $c\notin L(v)$, then $d_w(v,c)=0$.

Applying Theorem~\ref{thm:kahn} with $X=V(G)$,
we obtain a matching covering $N-o(N)$
original vertices.
For each matched edge $Q\cup\{c\}$, colour $Q$
with $c$. This properly colours $N-o(N)$ vertices
of $G$ from their lists, with each colour class
having size $1$ or $s$.
This contradicts our choice of counterexamples
and completes the proof.
\end{proof}

\subsection{From a given colouring to a list colouring}

We now prove Theorem~\ref{prop:lifting}.
The main idea is to use the given colouring to
reduce the number of lists containing each colour,
so that Lemma~\ref{lem:partial} can be applied.

\begin{proof}[Proof of Theorem~\ref{prop:lifting}]
Fix $s\ge1$ and $\varepsilon>0$, and choose
$0<\varepsilon'<\varepsilon/10$.
Let $\delta'>0$ be given by Lemma~\ref{lem:partial}
with parameter $s$ and error $\varepsilon'$.
Set $\delta=\varepsilon'\delta'$ and take $n$
sufficiently large.

Let $G$ be an $n$-vertex graph with a proper
colouring into $k$ classes, each of size at most $s+1$.
Suppose every vertex set of size greater than
$\delta n$ contains an independent set of size $s$.
Give each vertex an arbitrary list of
$k+\lceil4\varepsilon' n\rceil$ colours.

\smallskip
\noindent\textbf{Step 1: Reducing colour occurrences.}
The given colouring partitions $V(G)$ into $k$
independent sets.
As long as there is a remaining colour class $P$ and
a colour $c\in\bigcap_{v\in P}L(v)$,
colour all vertices of $P$ with $c$, remove $P$,
and delete $c$ from every remaining list.

If all vertices are coloured, we are done.
Otherwise, let $k'$ be the number of remaining
colour classes and $G'$ the subgraph induced by their union.
Since $k-k'$ distinct colours have been used,
each remaining list has at least
$k'+\lceil4\varepsilon' n\rceil$ colours.

By construction, no remaining class has a common
available colour. Since each class has at most
$s+1$ vertices, every colour occurs in at most
$s$ lists within each class, and hence in at most
$sk'$ remaining lists.
Shorten all remaining lists to exactly
$k'+\lceil4\varepsilon' n\rceil$ colours.

\smallskip
\noindent\textbf{Step 2: Reserving colours.}
We now set aside some colours to colour the vertices
that may remain uncoloured after applying
Lemma~\ref{lem:partial}.
Let $R$ be a random set obtained by including
each colour appearing in the remaining lists
independently with probability
$
2\varepsilon' n/(k'+\lceil4\varepsilon' n\rceil).
$
These colours will not be used until the final step.

Since every remaining list has exactly
$k'+\lceil4\varepsilon' n\rceil$ colours,
the random variable $|L(v)\cap R|$ is binomial
with mean $2\varepsilon' n$ for every $v\in V(G')$.
By the Chernoff bound and a union bound,
\[
\Prob\bigl(\exists v\in V(G'):
|L(v)\cap R|\notin[\varepsilon' n,3\varepsilon' n]\bigr)
\le2n\exp(-\varepsilon' n/6)=o(1).
\]

Thus we may fix $R$ such that, simultaneously
for every $v\in V(G')$, \vspace{-0.1cm}
\begin{equation}\label{eq:reserve}
\varepsilon' n\le|L(v)\cap R|\le3\varepsilon' n.\vspace{-0.1cm}
\end{equation}
Consequently, every vertex $v\in V(G')$ has at least\vspace{-0.1cm}
\[
|L(v)\setminus R|
\ge k'+\lceil4\varepsilon' n\rceil
-\lfloor3\varepsilon' n\rfloor
\ge k'+\lceil\varepsilon' n\rceil\vspace{-0.1cm}
\]
available colours outside $R$.

\smallskip
\noindent\textbf{Step 3: Partial colouring and completion.}
Shorten each list outside $R$ to exactly
$k'+\lceil\varepsilon' n\rceil$ colours.
Every colour outside $R$ occurs in at most
$sk'\le s(k'+\lceil\varepsilon' n\rceil)$ lists.
Moreover, since
$\delta'(k'+\lceil\varepsilon' n\rceil)\ge \delta n$,
every vertex set of $G'$ of size greater than
$\delta'(k'+\lceil\varepsilon' n\rceil)$ contains
an independent set of size $s$.

Thus Lemma~\ref{lem:partial} gives a proper
partial colouring of $G'$ leaving at most
$\varepsilon' n$
vertices uncoloured.
By~\eqref{eq:reserve}, each uncoloured vertex
has at least $\varepsilon' n$ colours in $R$.
Since no colour in $R$ has been used, we can
colour these vertices greedily with distinct
reserved colours.

Together with the colouring from Step~1,
this gives a proper colouring of $G$ from
arbitrary lists of size $k+\lceil4\varepsilon' n\rceil$.
Since $4\varepsilon'<\varepsilon$, for sufficiently large $n$,
$
\ch(G)\le k+\lceil4\varepsilon' n\rceil
\le k+\varepsilon n.
$
\end{proof}

\section{Proof of the main theorem}\label{sec:proof}

Proposition~\ref{prop:growing} and Theorem~\ref{prop:lifting}
provide the two colouring arguments.
To apply the latter to $G(n,p)$, we need to control
both the vertices belonging to large independent sets
and the occurrence of smaller independent sets in
every sufficiently large vertex set.

\subsection{Independent sets of fixed size in random graphs}

Note that the expected number of independent sets of size $s$
in $G(n,p)$ is $\binom ns(1-p)^{\binom s2}$, which is linear
in $n$ when $(1-p)n^{2/s}=\Theta(1)$.
The following lemma shows that below this scale, only $o(n)$
vertices belong to such an independent set, whereas above it,
every vertex set containing a fixed positive fraction of the
vertices contains one.

\begin{lemma}\label{lem:random-independent-sets}
Fix an integer $s\ge2$, and let $G=G(n,p)$, where
$p=p(n)\in(0,1)$. Let $Z$ be the set of vertices
belonging to an independent set of size $s$ in $G$.
If $(1-p)n^{2/s}\to0$, then with high probability
$|Z|=o(n)$.
If $(1-p)n^{2/s}\to\infty$, then, for every fixed
$\varepsilon>0$, with high probability every vertex set
of size at least $\varepsilon n$ contains an independent
set of size $s$.
\end{lemma}

The first assertion follows from the first moment
method. For the second, we use the following form
of Janson's inequality.

\begin{lemma}[\L uczak and Ruci\'nski~{\cite[Theorem~2.14]{JLR}}]
\label{lem:janson}
Let $\Omega$ be a finite set, and let $\Omega_\theta$ include
each element independently with probability $\theta\in[0,1]$.
Let $\mathcal A$ be a family of nonempty subsets of $\Omega$,
and let $I_A$ be the indicator of $A\subseteq\Omega_\theta$.
Put $X=\sum_{A\in\mathcal A}I_A$ and $\mu=\E X>0$. Then
\[
\Prob(X=0)\le\exp\left(-\frac{\mu^2}{2\bar\Delta}\right),
\qquad
\bar\Delta=\sum_{\substack{A,B\in\mathcal A\\A\cap B\ne\varnothing}}
\E[I_AI_B],\vspace{-0.2cm}
\]
where the sum is over ordered pairs, including $A=B$.
\end{lemma}

\begin{proof}[Proof of Lemma~\ref{lem:random-independent-sets}]
\noindent\textbf{Case 1: $(1-p)n^{2/s}\to0$.}
Let $N_s$ count the independent sets of size $s$ in $G$.
\[
\E|Z|
\le s\E N_s
=s\binom ns(1-p)^{\binom s2}
\le sn\bigl((1-p)n^{2/s}\bigr)^{\binom s2}
=o(n).
\]
Markov's inequality gives $|Z|=o(n)$ with high probability.

\smallskip
\noindent\textbf{Case 2: $(1-p)n^{2/s}\to\infty$.}
Fix $\varepsilon\in(0,1)$ and a set
$U\subseteq[n]$ of size $m=\lceil\varepsilon n\rceil$.
Let $X$ count the independent sets of size $s$ in $G[U]$.
Apply Lemma~\ref{lem:janson} to the random set of absent
pairs in $\binom U2$, each of which is absent independently
with probability $1-p$.
For each $s$-vertex set, the required absent pairs are
all pairs of its vertices. Thus
$\mu=\binom ms(1-p)^{\binom s2}$.

Two such sets of pairs intersect precisely when the
corresponding vertex sets share at least two vertices.
Counting ordered pairs according to the number
$j$ of common vertices gives
\[
\bar\Delta
=\sum_{j=2}^{s}
\binom ms\binom sj\binom{m-s}{s-j}
(1-p)^{\,2\binom s2-\binom j2}.
\]
Consequently, \vspace{-0.2cm}
\[
\frac{\bar\Delta}{\mu^2}
=\sum_{j=2}^{s}
\frac{\binom sj\binom{m-s}{s-j}}{\binom ms}
(1-p)^{-\binom j2}
\le C_{s,\varepsilon}
\sum_{j=2}^{s}\frac1{n^j(1-p)^{\binom j2}}
=o(1/n).
\]
Here the last estimate follows from
$n^{j-1}(1-p)^{\binom j2}\to\infty$ for every
$2\le j\le s$, since $(1-p)n^{2/s}\to\infty$.
By Janson's inequality and a union bound over all
$m$-vertex sets, for sufficiently large $n$,
\[
\Prob\bigl(\exists U\subseteq[n]:\ |U|=m,
\ \alpha(G[U])<s\bigr)
\le 2^n e^{-2n}=o(1).
\]
Every vertex set of size at least $\varepsilon n$
contains an $m$-vertex set, proving the second assertion.
\end{proof}

\subsection{Combining the two regimes}

\begin{proof}[Proof of Theorem~\ref{thm:main}]
Let $G=G(n,p)$ and write $d=np$.
By the subsubsequence principle, we may assume that
\[
\frac{2\log d}{\log(1/(1-p))}
\longrightarrow a\in[0,\infty].
\]
If $a=\infty$, Proposition~\ref{prop:growing}
applies. Henceforth assume $a<\infty$.

\smallskip
\noindent\textbf{Step 1: Choosing the colour-class bound.}
Since $d\to\infty$ and $a<\infty$, we have
$p\to1$ and $d=(1+o(1))n$.
Set $s=\max\{1,\lfloor a\rfloor-1\}$.
Then $a<s+2$, and $a>s$ if $s\ge2$.
For every fixed integer $t\ge2$,\vspace{-0.1cm}
\[
\log((1-p)n^{2/t})
=\left(\frac at-1+o(1)\right)
\log\frac1{1-p}.
\]
Since $\log(1/(1-p))\to\infty$, we obtain\vspace{-0.1cm}
\begin{equation}\label{eq:window}
\begin{aligned}
&(1-p)n^{2/(s+2)}\to0,\ \ \ \ \text{and if }s\ge2,\quad
(1-p)n^{2/s}\to\infty.
\end{aligned}
\end{equation}

\smallskip
\noindent\textbf{Step 2: Removing exceptional vertices.}
Let $Z$ be the set of vertices belonging to an independent
set of size $s+2$ in $G$, and put $G_0=G-Z$.
By Lemma~\ref{lem:random-independent-sets} and the first
limit in~\eqref{eq:window}, $|Z|=o(n)$ with high probability.
By construction, $\alpha(G_0)\le s+1$, and
$|V(G_0)|=n-o(n)$ with high probability.

For every fixed $\gamma>0$, the second limit and Lemma~\ref{lem:random-independent-sets} imply that, with high probability, every vertex set
of $G$ of size at least $\gamma n$ contains an independent
set of size $s$ when $s\ge2$.
For $s=1$, every nonempty vertex set already contains
an independent set of size one.

\smallskip
\noindent\textbf{Step 3: Completing the colouring.}
Fix $\varepsilon>0$, and let $\delta>0$ be given by
Theorem~\ref{prop:lifting} for $s$ and $\varepsilon$.
Step~2 applied with $\gamma=\delta/2$ ensures that,
with high probability, every vertex set of $G_0$ of size
greater than $\delta|V(G_0)|$ contains an independent
set of size $s$.
An optimal colouring of $G_0$ has $\chi(G_0)$ colour
classes, each of size at most $s+1$.
Thus Theorem~\ref{prop:lifting} gives
$\ch(G_0)\le\chi(G_0)+\varepsilon|V(G_0)|
\le\chi(G_0)+\varepsilon n$
with high probability.
As $\varepsilon$ is arbitrary,
$\ch(G_0)\le\chi(G_0)+o(n)$ with high probability.

To colour $G$, first assign distinct allowed colours
to the vertices of $Z$ and delete these colours
from all remaining lists.
Starting with lists of size $|Z|+\ch(G_0)$ leaves
at least $\ch(G_0)$ colours at every vertex of $G_0$.
Consequently, with high probability, \vspace{-0.1cm}
\[
\ch(G)\le |Z|+\ch(G_0)\le\chi(G)+o(n).
\]

Since $\alpha(G_0)\le s+1$ and $|Z|=o(n)$,\vspace{-0.2cm}
\[
\chi(G)\ge\chi(G_0)\ge\frac{n-|Z|}{s+1}
=\Omega(n).
\]
Thus the additive error $o(n)$ is also
$o(\chi(G))$, completing the proof.
\end{proof}

\section{Concluding remarks}

Our main theorem establishes the asymptotic equality
of the choice number and the chromatic number of
$G(n,p)$ whenever $np\to\infty$.
Several natural questions remain.

\medskip
\noindent\textbf{The gap between the two parameters.}
Our result determines the choice number up to a
$1+o(1)$ factor, but does not quantify the difference
between the two parameters.
As already suggested by Alon, Krivelevich and
Sudakov~\cite{AKS}, it is conceivable that
they are in fact equal with high probability.

\medskip
\noindent\textbf{Online list colouring.}
A natural strengthening concerns the online choice
number $\ch_{\mathrm{on}}(G)$, also known as the
paint number.
Frieze, Mitsche, P\'erez-Gim\'enez and
Pra{\l}at~\cite{FMPP2015} established asymptotic
equality between the online choice number and
the chromatic number in a substantial range of
edge probabilities.
It remains interesting to determine whether
$
\ch_{\mathrm{on}}(G(n,p))
=(1+o(1))\chi(G(n,p))
$
holds with high probability whenever $np\to\infty$.

\medskip

\noindent\textbf{Random  hypergraphs.}
Krivelevich and Vu~\cite{KV} proposed the following
hypergraph analogue. Let $H_\ell(n,p)$ denote
the random $\ell$-uniform hypergraph.

\begin{conjecture}[\cite{KV}]
For every fixed integer $\ell\ge3$, if
$n^{\ell-1}p\to\infty$, then with high probability,
\[
\ch(H_\ell(n,p))=(1+o(1))\chi(H_\ell(n,p)).
\]
\end{conjecture}

We expect that our methods can also be extended
to establish this conjecture.
To keep the presentation concise, we restrict
the detailed proofs to graphs and briefly outline
the necessary modifications for random 
hypergraphs.
\smallskip

For the large-independent-set regime, the required
chromatic-number asymptotics follow from Krivelevich
and Sudakov~\cite{KS} when $p=o(1)$, and from standard
first-moment and Janson arguments when $p$ is bounded
away from zero.
In the proof of Lemma~\ref{lem:uniform-independent},
Bollob\'as's inequality
extends to uniform hypergraphs by the same argument,
with the exponent
$\binom{t}{2}/\binom{m}{2}$
replaced by
$\binom{t}{\ell}/\binom{m}{\ell}$.
The bounded-differences argument remains valid,
and the completion lemma of Krivelevich and
Vu~\cite{KV} already applies to hypergraphs.
The remainder of the argument carries over unchanged.

\smallskip

For the bounded-independent-set regime,
Theorem~\ref{prop:lifting} and its proof extend
without change to arbitrary hypergraphs, since
they rely only on independent colour classes
and Kahn's matching theorem for an auxiliary
hypergraph.
The main modification concerns
Lemma~\ref{lem:random-independent-sets}.
For every fixed $s\ge\ell$, the threshold
$(1-p)n^{2/s}$ is replaced by
$
(1-p)n^{(s-1)/\binom{s}{\ell}},
$
since an $s$-vertex set is independent with
probability $(1-p)^{\binom{s}{\ell}}$.
The first-moment argument remains unchanged,
while the Janson argument now counts pairs of
$s$-sets intersecting in at least $\ell$ vertices
rather than two. 
As in the proof of Theorem~\ref{thm:main},
we choose an integer $s\ge\ell-1$ such that
$
(1-p)n^{(s+1)/\binom{s+2}{\ell}}\to0,
$
and, when $s\ge\ell$,
$
(1-p)n^{(s-1)/\binom{s}{\ell}}\to\infty.
$
Consequently, with high probability, deleting
$o(n)$ exceptional vertices leaves a hypergraph
with independence number at most $s+1$,
in which every vertex subset of linear size
contains an independent set of size $s$.
Theorem~\ref{prop:lifting} therefore applies,
and the remainder of the proof carries over unchanged.

\section*{Acknowledgements} Lanchao Wang was supported by the NSFC under grant number 12471327, the
National Key R\&D Program of China under grant number 2024YFA1013900,
the China Scholarship Council, and the Institute for Basic Science
(IBS-R029-C4). 
Jiasheng Zeng acknowledges support from the Research Grants Council of Hong Kong through the Hong Kong PhD Fellowship Scheme.

In this work, ChatGPT assisted with mathematical derivations, discussions, and the preparation of an early draft. The proof of the large-independent-set case was developed autonomously by GPT-5.6. The overall proof strategy for the bounded-independent-set case, particularly the colouring method, was developed by the authors, with GPT-6 assisting in working out the technical details.
The final manuscript was written by the authors,
who take full responsibility.

\end{document}